\documentclass{amsart}
\usepackage{adjustbox,lipsum}
\usepackage{mathtools}
\usepackage{amsmath}
\usepackage{amssymb}
\usepackage[shortlabels]{enumitem}
\usepackage{colonequals}
\usepackage{tabu}
\usepackage{tikz-cd}
\usepackage{mathrsfs} % for \mathscr (script letters)
\usepackage{makecell}

\usepackage[
        colorlinks, citecolor=darkgreen,
        backref,
        pdfauthor={Samir Siksek},
]{hyperref}

\usepackage{comment}
\newcommand{\PP}{\mathbb{P}}
\newcommand{\Q}{\mathbb{Q}}

\newcommand{\Z}{\mathbb{Z}}
\newcommand{\cX}{\mathscr{X}}
\newcommand{\cL}{\mathcal{L}}

\newcommand{\cF}{\mathcal{F}}

\newcommand{\OO}{\mathcal{O}}

\DeclareMathOperator{\Gal}{Gal}

\renewcommand{\setminus}{-}

\newtheorem{thm}{Theorem}

\newtheorem{lem}[thm]{Lemma}

\newtheorem{conj}[thm]{Conjecture}

\theoremstyle{definition}

\theoremstyle{remark}

\definecolor{darkgreen}{rgb}{0,0.5,0}

\begin{document}

\title[]{
On the unit equation $\varepsilon+\delta=n$ in cubic fields
}

\begin{abstract}
Let $n$ be an integer not equal to $-2$, $0$ or $2$.
We consider the equation $\varepsilon+\delta=n$ in units $\varepsilon$, $\delta$ of cubic fields.
We show that this unit equation has no solutions for $100\%$ of cubic fields, when ordered by discriminant.
This is consistent with a recent conjecture of the authors.
\end{abstract}

\author{Maleeha Khawaja}

\address{Mathematics Institute\\
    University of Warwick\\
    CV4 7AL \\
    United Kingdom}

\email{maleeha.khawaja@warwick.ac.uk}

\author{Samir Siksek}

\address{Mathematics Institute\\
    University of Warwick\\
    CV4 7AL \\
    United Kingdom}

\email{s.siksek@warwick.ac.uk}

\date{\today}


\maketitle

\section{Introduction}

Let $K$ be a number field, and write $\OO_K$ for its ring of integers.
Let $C$ be a smooth projective and absolutely irreducible curve of genus $g$, defined over
$K$.
Let $D$ be a reduced effective divisor
on $C$ and consider the punctured curve $X=C \setminus D$. The
Euler characteristic of $X$ is given by $\chi(X)=2-2g-\deg(D)$. We say
that $X$ is hyperbolic if its Euler characteristic is negative. 
Let $\cX/\OO_K$ be
a model for $X$ over $\OO_K$.
The famous Faltings--Siegel Theorem (Remark D.9.2.2 of \cite{MR1745599})
asserts that $\cX(\OO_K)$ is finite. 
Now let $L/K$ be a finite extension.
We define the set of \textbf{$L$-new points} by
\[
\cX(\OO_L)_\mathrm{new}=\{ P \in \cX(\OO_L) \; : \; K(P)=L\}.
\]
We note that the $L$-new points are ones defined over $L$,
but not over any strictly smaller extension of $K$.
In \cite{KS_stab} we state the following conjecture, which we refer to as the
\textbf{statistical Diophantine stability conjecture for hyperbolic curves}.
\begin{conj}\label{conj:diostab}
Let $K$ be a number field and $\cX/\OO_K$ a model
for a hyperbolic curve. Let $\cL$ be an allowable family of extensions $L/K$.
Then $\cX(\OO_L)_{\mathrm{new}}=\emptyset$ for 100\% of
$L \in \cL$, when ordered by norms of discriminants.
\end{conj}
For the definition of an allowable family of extensions we refer to \cite{KS_stab},
but for the purpose of this paper we note that it includes the family
of all extensions $L/K$ of a given fixed degree $m \ge 2$. The following result 
is Theorem 7 of \cite{KS_stab}.
\begin{thm}\label{thm:original}
	$(\PP^1 \setminus \{0,1,\infty\})(\OO_L)=\emptyset$ for $100\%$
	of cubic fields $L$, when ordered by absolute discriminant.
\end{thm}

In this paper we give further evidence of the statistical Diophantine
stability conjecture by proving the following generalization.

\begin{thm}\label{thm:main}
Let $n$ be a non-zero integer. 
	\begin{enumerate}[(I)]
		\item There are infinitely many cubic fields $L$ such that
			$(\PP^1 \setminus \{0,n,\infty\})(\OO_L)_{\mathrm{new}} \ne \emptyset$.
		\item $(\PP^1 \setminus \{0,n,\infty\})(\OO_L)_{\mathrm{new}}=\emptyset$ for $100\%$
		of cubic fields $L$, when ordered by absolute discriminant.
	\end{enumerate}
\end{thm}
By contrast to (I), for any non-zero integer $n$, there are finitely many quadratic
fields $L$ for which $(\PP^1 \setminus \{0,n,\infty\})(\OO_L)_{\mathrm{new}} \ne \emptyset$;
this is shown in Section~\ref{sec:quadratic}. Thus the quadratic analogue of (II)
trivially holds.

Let $L/\Q$ be a number field. Let $\varepsilon \in \PP^1(L)$.
We note that $\varepsilon \in (\PP^1 \setminus \{0,n,\infty\})(\OO_L)$
if and only if it does not reduce to any of $0$, $n$, $\infty$
modulo any prime ideal of $\OO_L$. This is equivalent to
$\varepsilon$ and $n-\varepsilon$ being units of $\OO_L$.
Thus,
\[
	(\PP^1 \setminus \{0,n,\infty\})(\OO_L) \; = \; \{\varepsilon \in \OO_L^\times \; : \; (n-\varepsilon) \in \OO_L^\times\}.
\]
We see that there is a $1$-$1$ correspondence between $\OO_L$-points of $\PP^1 \setminus \{0,n,\infty\}$
and solutions to the unit equation
\begin{equation}\label{eqn:unit}
	\varepsilon+\delta \; = \; n, \qquad \varepsilon,~\delta \in \OO_L^\times.
\end{equation}
The finiteness of solutions to the unit equation is a famous theorem of Siegel \cite{Siegel},
now often viewed as a special case of the Faltings--Siegel theorem.

We note that $(-\varepsilon,-\delta)$ is a solution to the corresponding unit equation
with $n$ replaced by $-n$. Henceforth we suppose, without loss of generality, that $n$ is a positive integer.

\medskip

We briefly discuss previous works which study the unit equation~\eqref{eqn:unit}. 
There are numerous works which study~\eqref{eqn:unit} for $n=1$.
In a series of papers spanning 
forty years and culminating in \cite{Nagell2,Nagell}, Nagell provided a 
complete classification of solutions to \eqref{eqn:unit} over number fields of unit rank 
at most one. We refer the reader to the work of Evertse and Gyory~\cite{EvertseGyory}, 
which contains a comprehensive survey of effective methods and results surrounding~\eqref{eqn:unit}.

For general $n$, most papers focus on fixing a number field $L$ and determining 
the set of integers $\mathcal{N}_{L}$ for which there is a solution to~\eqref{eqn:unit} over $L$; 
see, for example, \cite{Jarden07}, \cite{Kostra}, \cite{Newman}. 
In particular, we point out recent work of Tinkov\'{a}, Visser and Yatsyna~\cite{TVY26} which 
provides an explicit upper bound for $\mathcal{N}_{L}$ for any number field $L$ 
not containing a real quadratic subfield. 
The aforementioned work also gives a parametrization of all solutions 
to~\eqref{eqn:unit} when $L$ is a complex cubic field or a cyclic cubic field.

\section{Quadratic Solutions to the Unit Equation \eqref{eqn:unit}}\label{sec:quadratic}

\begin{thm}
Let $n$ be a positive integer.
	The only solutions to \eqref{eqn:unit} in quadratic fields, up to Galois conjugation and swapping $\varepsilon$ and $\delta$,
	are
	\begin{itemize}
		\item for $n=1$,
		\[
			\left(\frac{1+\sqrt{-3}}{2}, \, \frac{1-\sqrt{-3}}{2} \right),
			\qquad
			\left(\frac{1+\sqrt{5}}{2}, \frac{1-\sqrt{5}}{2} \right),
			\qquad
			\left(\frac{-1+\sqrt{5}}{2}, \frac{3-\sqrt{5}}{2} \right);
		\]

		\item for $n=2$, 
		\[
			(1,1), \qquad (1+\sqrt{2},1-\sqrt{2}), \qquad
			\left(\frac{1+\sqrt{5}}{2}, \frac{3-\sqrt{5}}{2} \right);
		\]
		\item for $n \ge 3$,
		\[
			\left(\frac{n + \sqrt{n^2-4 }}{2}, 
			\frac{n - \sqrt{n^2-4 }}{2} \right), 
			\qquad
			\left(\frac{n + \sqrt{n^2+4 }}{2}, 
			\frac{n - \sqrt{n^2+4}}{2} \right).
		\]
	\end{itemize}
\end{thm}
\begin{proof}
Let $(\varepsilon,\delta)$ be a solution to \eqref{eqn:unit} with $L$ quadratic.
We note that if $\varepsilon \in \Q$, then $\varepsilon$, $\delta = \pm 1$, and so $n=2$,
	and $(\varepsilon,\delta)=(1,1)$.

We may therefore suppose that $L=\Q(\varepsilon)$. Write $m_\varepsilon(X)$ and $m_\delta(X)$
for the minimal polynomials of $\varepsilon$ and $\delta$ respectively. 
	Then,
	\[
		m_\varepsilon(X)=X^2+aX+N(\varepsilon), \qquad m_\delta(X)=X^2+bX+N(\delta)
	\]
	where $a$, $b \in \Z$, and $N(\varepsilon)$, $N(\delta)$ are the norms of $\varepsilon$
	and $\delta$ respectively, which both must be $\pm 1$. However, $\delta$ is a root
	of 
	\[
		m_\varepsilon(n-X)=X^2-(2n+a)X+n^2+an+N(\varepsilon).
	\]
	This must be equal to $m_\delta(X)$ and therefore
	\begin{equation}\label{eqn:quad}
		n^2+an=N(\delta)-N(\varepsilon).
	\end{equation}
	Suppose $n \ge 3$ first. It follows that $N(\delta)=N(\varepsilon)$ and $a=-n$.
	Thus $\varepsilon$ is a root of $X^2-nX+\alpha$ where $\alpha=\pm 1$. Hence, up to
	Galois conjugation,
	\[
		\varepsilon=\frac{n + \sqrt{n^2-4 \alpha}}{2}, 
		\qquad
		\delta=\frac{n - \sqrt{n^2-4 \alpha}}{2}. 
	\]
	It is easy to check that $n^2-4$ and $n^2+4$ are never squares for $n \ge 3$, so these are genuinely quadratic
	solutions.

	Suppose now that $n=2$. In addition to the case $N(\delta)=N(\varepsilon)$ considered
	above, we must also consider $N(\delta)=-N(\varepsilon)$. By swapping $\varepsilon$
	and $\delta$ if necessary, we may suppose that $N(\delta)=1$ and $N(\varepsilon)=-1$. 
	From \eqref{eqn:quad}, we obtain $a=-1$, and thus $\varepsilon$ is a root of $X^2-X-1$.
	Hence, up to Galois conjugation, 
	\[
		\varepsilon=\frac{1+\sqrt{5}}{2}, \qquad \delta=\frac{3-\sqrt{5}}{2}.
	\]

	Finally suppose $n=1$. This case was known to Nagell~\cite[Section 2]{Nagell2} but we include 
	the details for completeness. Again, we must additionally consider the case $N(\delta)=-N(\varepsilon)$ 
	and we can assume that, without loss of generality, $N(\delta)=1$ and $N(\varepsilon)=-1$. 
	From \eqref{eqn:quad}, we obtain $a=1$, and thus $\varepsilon$ is a root of $X^2+X-1$. 
	Therefore, up to Galois conjugation,
	\[
		\varepsilon=\frac{-1+\sqrt{5}}{2}, \qquad \delta=\frac{3-\sqrt{5}}{2}.
	\]

\end{proof}

\section{Cubic Solutions to the Unit Equation \eqref{eqn:unit}}
In this section we give an explicit parametrization of solutions
of the unit equation \eqref{eqn:unit} that generate cubic fields.
\begin{lem}\label{lem:minpol}
Let $n \ge 3$. Let $(\varepsilon,\delta)$ be a solution
	to \eqref{eqn:unit} with $L=\Q(\varepsilon)$ being
	a cubic field. Then,
	after possibly swapping $\varepsilon$ and $\delta$,
	there is some $a \in \Z$
	such that the minimal polynomial
	of $\varepsilon$ is 
	\begin{equation}\label{eqn:minpol3}
		X^3+a X^2-(n^2+an)X+1.
	\end{equation}
	Conversely, let $a \in \Z$.  Then the polynomial
	\eqref{eqn:minpol3} is irreducible. Let
	$\varepsilon$ be
	a root of \eqref{eqn:minpol3},
	and write $\delta=n-\varepsilon$. Then $(\varepsilon,\delta)$
	is a solution to \eqref{eqn:unit} with $L=\Q(\varepsilon)$
	a cubic field.
\end{lem}
\begin{proof}
	Note that the minimal polynomials of $\varepsilon$
	and $\delta$ have the form
	\begin{equation}\label{eqn:ms}
		m_\varepsilon(X)=X^3+aX^2+bX-N(\varepsilon), \qquad m_\delta(X)=X^3+cX^2+dX-N(\delta),
	\end{equation}
	where $a$, $b$, $c$, $d \in \Z$, and $N(\varepsilon)$, $N(\delta)$ are the norms of $\varepsilon$
	and $\delta$ respectively, which both must be $\pm 1$. 
	However, $\delta$ is a root $m_\varepsilon(n-X)$ whose leading
	coefficient is $-1$. We conclude that $m_\delta(X)=-m_\varepsilon(n-X)$.
	Comparing constant coefficients we obtain
	\[
		n(n^2+an+b)=N(\varepsilon)+N(\delta).
	\]
	As $n \ge 3$, we see that $N(\varepsilon)$ and $N(\delta)$
	are opposite signs, and after possibly swapping them
	we may suppose that $N(\varepsilon)=-1$ and $N(\delta)=1$.
	Thus $b=-n^2-an$, proving that \eqref{eqn:minpol3}
	is the minimal polynomial of $\varepsilon$.

	For the converse let $a \in \Z$. We first show that
	\eqref{eqn:minpol3} is irreducible. If it is not then
	$\pm 1$ is a root, whence
	\begin{equation}\label{eqn:cases}
		n^2+an-(2+a)=0, \qquad \text{or} \qquad n^2+an+a=0.
	\end{equation}
	Let $m=2n+a$. Then
	\[
		(m+a+2)(m-a-2)=4, \qquad \text{or} \qquad
		(a-2+m)(a-2-m)=4,
	\]
	according to whether we are in the first or second case of 
	\eqref{eqn:cases}. In either case, both factors
	have the same parity, and so both are equal to $2$
	or both are equal to $-2$.
	We quickly conclude that $n=0$, $2$ or $-2$, giving a contradiction.
	Finally, write $f$ for the polynomial in \eqref{eqn:minpol3}
	and let $\varepsilon$ be a root of $f$,
	and let $\delta=n-\varepsilon$. Then $\delta$ is a root 
	of $-f(n-X)$ whose constant coefficient is 
	$-f(n)=-1$, so $\delta$ is a unit.
\end{proof}

\begin{lem}\label{lem:minpol2}
	Let $n=2$. Let $(\varepsilon,\delta)$ be a solution
		to \eqref{eqn:unit} with $L=\Q(\varepsilon)$ being
		a cubic field. Then,
		after possibly swapping $\varepsilon$ and $\delta$,
		there is some $a \in \Z$
		such that the minimal polynomial
		of $\varepsilon$ is one of
		\begin{equation}\label{eqn:minpol2}
			\begin{split}
			X^3+aX^2-(2a+3)X-1; \\
			X^3+aX^2-(2a+4)X+1; \\
			X^3+aX^2-(2a+5)X+1.
			\end{split}
		\end{equation}
		Conversely, let $a \in \Z$, $a\neq -2$, $-3$, and let $\varepsilon$
		be a root of one of the polynomials in \eqref{eqn:minpol2} (which are irreducible),
		and write $\delta=2-\varepsilon$. Then $(\varepsilon,\delta)$
		is a solution to \eqref{eqn:unit} with $L=\Q(\varepsilon)$
		a cubic field.
	\end{lem}
\begin{comment}
	\begin{proof}
		Recall from the proof of Lemma~\ref{lem:minpol} the minimal
		polynomials of $\varepsilon$ and $\delta$ are of the form 
		\eqref{eqn:ms} with $a$, $b$, $c$, $d \in \Z$,
		and that moreover,
		\begin{equation}\label{eq:nequals2}
			2(4+2a+b)=N(\varepsilon)+N(\delta).
		\end{equation}
		There are three cases to consider.
		\begin{itemize}
			\item If $N(\varepsilon)=N(\delta)=1$ then $b=-2a-3$.
			\item If $N(\varepsilon)=-N(\delta)$ then, after possibly swapping $\varepsilon$ and $\delta$, 
			we can suppose $N(\varepsilon)=-1$ without loss of generality, and so $b=-2a-4$. 
			\item If $N(\varepsilon)=N(\delta)=-1$ then $b=-2a-5$.
		\end{itemize}

			Let $f$ be one of the polynomials given in~\eqref{eqn:minpol2}.
			The proof of the converse statement is similar to Lemma~\ref{lem:minpol}.
	\end{proof}
\end{comment}
\begin{lem}\label{lem:minpol1}
	Let $n=1$. Let $(\varepsilon,\delta)$ be a solution
		to \eqref{eqn:unit} with $L=\Q(\varepsilon)$ being
		a cubic field. Then,
		after possibly swapping $\varepsilon$ and $\delta$,
		there is some $a \in \Z$
		such that the minimal polynomial
		of $\varepsilon$ is one of
		\begin{equation}\label{eqn:minpol1}
			\begin{split}
			X^3+aX^2-(a+1)X+1; \\
			X^3+aX^2-(a+3)X+1; \\
			X^3+aX^2-(a-1)X-1.
			\end{split}
		\end{equation}
		Conversely, let $a \in \Z$, and let $\varepsilon$
		be a root of one of the polynomials in \eqref{eqn:minpol2} (which are irreducible),
		and write $\delta=1-\varepsilon$. Then $(\varepsilon,\delta)$
		is a solution to \eqref{eqn:unit} with $L=\Q(\varepsilon)$
		a cubic field.
	\end{lem}
The proofs for $n=2$ and $n=1$ are similar to the previous cases. We note that Lemma~\ref{lem:minpol1} was known to Nagell~\cite{Nagell2}. In Nagell's paper
only the first two cases of \eqref{eqn:minpol1} are given. This is because we can go from the 
third case to the first case by observing that if $(\varepsilon,\delta)$ is a solution to the
unit equation with $n=1$, then so is $(-\delta/\varepsilon,1/\varepsilon)$.
\begin{comment}
	\begin{proof}
		This case was known to Nagell~\cite{Nagell2} but we include the proof for completeness. 
		Recall from the proof of Lemma~\ref{lem:minpol} the minimal
		polynomials of $\varepsilon$ and $\delta$ are of the form 
		\eqref{eqn:ms} with $a$, $b$, $c$, $d \in \Z$,
		and that moreover,
		\begin{equation}
			1+a+b=N(\varepsilon)+N(\delta).
		\end{equation}
		There are three cases to consider.
		\begin{enumerate}[(a)]
			\item If $N(\varepsilon)=-N(\delta)$ then, after possibly swapping $\varepsilon$ and $\delta$, 
			we can suppose $N(\varepsilon)=-1$ without loss of generality, and so $b=-a-1$. 
			\item If $N(\varepsilon)=N(\delta)=-1$ then $b=-a-3$.
			\item Finally, suppose $N(\varepsilon)=N(\delta)=1$. 
				  By dividing through by $\varepsilon$, we can rewrite~\eqref{eqn:unit} as 
				  \[
				 	\frac{1}{\varepsilon} + \left(-\frac{\delta}{\varepsilon}\right) = 1. 
				  \]
				This takes us to case (a) as $N(1/\varepsilon)=1$ and $N(-\delta/\varepsilon)=-1$.
		\end{enumerate}
			Let $f$ be one of the polynomials given in~\eqref{eqn:minpol1}.
			The proof of the converse statement is similar to Lemma~\ref{lem:minpol}.
	\end{proof}
\end{comment}

\section{Proof of Theorem~\ref{thm:main}}
Part (I) of Theorem~\ref{thm:main} follows trivially from
Lemmas~\ref{lem:minpol},~\ref{lem:minpol2},~\ref{lem:minpol1}
and the fact that over each number field $L$, the unit equation \eqref{eqn:unit}
has finitely many solutions. 

\medskip

The rest of this section is devoted to proving part (II) of Theorem~\ref{thm:main}.
We now recall the following recent result of the authors \cite[Corollary 19]{KS_stab}. 
A key ingredient in the proof of this is a 
recent theorem due to Bhargava, Taniguchi and Thorne~\cite[Theorem 1.3]{BTT} 
which counts cubic fields having local specifications 
at a finite set of primes.

\begin{thm}\label{thm:sparsecubics}
	Let $F(T,X) \in \Q[T,X]$ be irreducible and have degree $3$ in $X$. 
		Let $\Theta$ be the set of $t \in \Q$ such that $F(t,X)$ 
		either has degree $<3$ or is reducible.
	Let $\Delta_X(F) \in \Q[T]$
	be the discriminant of $F$ with respect to $X$, and write this as
	\[
			\Delta_X(F) \; = \; g(T) \cdot h(T)^2
	\]
	where $g$, $h \in \Q[T]$ with $g$ squarefree.
	Suppose $g$ has even degree. Let $G=\Gal(g)$. Suppose that there is an element $\sigma \in \Gal(g)$
	acting freely on the roots of $g$. Let $\cF^\prime$ be the set of cubic fields
	we obtain from $F(t,X)=0$ with $t \in \Q \setminus \Theta$.
	Then $100\%$ of cubic fields, ordered by discriminant,
	do not belong to the family $\cF^\prime$.
\end{thm}

Let $L$ be a cubic field, and
suppose that the punctured curve $\PP^1\setminus \{0,n,\infty\}$
has a $L$-new point.
Then, there is a solution $(\varepsilon,\delta)$ to \eqref{eqn:unit}
with $L=\Q(\varepsilon)=\Q(\delta)$. 

We first suppose that $n \ge 3$. 
After possibly swapping
$\varepsilon$ and $\delta$, we may suppose that $\varepsilon$
is a root of \eqref{eqn:minpol3}. Let
\[
	F(T,X)\; = \; X^3+T X^2-(n^2+nT)X+1.
\]
It follows from Lemma~\ref{lem:minpol} that $L$ belongs to 
$\cF^\prime$. To prove Theorem~\ref{thm:main} it is enough
to check the criteria of Theorem~\ref{thm:sparsecubics}.
We note that 
\[
	\Delta_X(F) \; = \; 
	n^2 T^4 + (6 n^3 - 4) T^3 + (13 n^4 - 18 n) T^2 + (12 n^5 - 18 n^2) T + 4 n^6 - 27.
\]
The discriminant with respect to $T$ of $\Delta_X(F)$ is
\[ (n^2 - 3 n + 3)^3
    (n^2 + 3)^3
    (n^2 + 3 n + 3)^3
\]
which is clearly non-zero. Thus $\Delta_X(F) \in \Q[T]$ is square-free
for all $n$, and we may, in the notation of Theorem~\ref{thm:sparsecubics},
take $g=\Delta_X(F)$, $h=1$.

We claim that $g$ does not have a rational root for all values of $n$.
Let
\[
	H(x,y,z) \; = \; 4 x^6 + 12 x^5 y + 13 x^4 y^2 + 6 x^3 y^3 + x^2 y^4 - 18 x^2 y z^3 -
    18 x y^2 z^3 - 4 y^3 z^3 - 27 z^6.
\]
This is a homogeneous degree $6$ polynomial that satisfies
$H(n,T,1)=g(T)$.
Suppose there is some $n \ge 3$ and some $t \in \Q$
such that $g(t)=0$. Then $(x:y:z)=(n:t:1)$ is a rational point on the 
projective plane curve
\[
	D \; : \; 	H(x,y,z) \; = \; 0.
\]
A quick search using \texttt{Magma}~\cite{Magma} reveals three rational points
on this model
\[
(0 : 1 : 0), \qquad
    (-1 : 2 : 0), \qquad
     (-1 : 1 : 0).
\]
All of these three points are singular, but each corresponds to a unique
rational point on the normalization $\tilde{D}$ of $D$.
We find that $\tilde{D}$ has genus $1$, and thus is an elliptic curve.
A Weierstrass model for $\tilde{D}$ is given by
\[
	E \; : \; Y^2 +Y \; = \; X^3,
\]
with Cremona label \texttt{27a3}.
We find that $E/\Q$ has rank $0$ and that  in fact $E(\Q) \cong \Z/3\Z$.
Thus the three rational points we have found on $\tilde{D}$ 
are the only ones. In particular, for all $n \ge 3$, the
quartic polynomial $g$ does not have a rational root. 

To complete the proof, using Theorem~\ref{thm:sparsecubics},
it enough to show that some element of $G=\Gal(g)$ acts freely on the four
roots of $g$. We identify $G$ with a subgroup of $S_4$.
As $g$ does not have a rational root, $G$ does not have a fixed point
in $\{1,2,3,4\}$.
By an elementary enumeration of the subgroups of $S_4$,
we see that every subgroup without a fixed point contains
an element that acts freely on $\{1,2,3,4\}$. This completes the proof for $n \ge 3$.
\begin{comment}
Since $g$ does not have a rational root, the only possible
partition of $\{1,2,3,4\}$ into $G$-orbits is one of,
\[
	\{\{1,2\},\{3,4\}\}, \qquad \{\{1,2,3,4\}\}.
\]

If $g$ is irreducible then $G=\Gal(g)$ contains 
an element that acts freely on the roots of $g$ by a theorem of 
Jordan~\cite{Jordan}. 
Otherwise, if $g$ is reducible, $G$ is a non-transitive subgroup of $S_4$. 
There are $6$ non-transitive subgroups of $S_4$, we list these below 
along with their generators:
\begin{multline*}
\{e\},\quad
C_2\cong \langle (1 \; 2)\rangle, \quad
C_2\cong \langle (1 \; 2)(3 \; 4) \rangle, \\
C_3\cong \langle (1 \; 2 \; 3) \rangle, \quad
C_2\times C_2\cong \langle (1 \; 2), (3 \; 4)\rangle, \quad
S_3\cong \langle (1 \; 2 \; 3), (1 \; 2)\rangle.
\end{multline*}
However, since $g$ does not have a rational root, any specified 
root of $g$ can not be fixed by all members of $G$. 
This leaves only two possibilities for $G$, i.e., 
\[
G\cong C_2\cong \langle (1 \; 2)(3 \; 4) \rangle \qquad \text{ or }\qquad G\cong C_2\times C_2\cong \langle (1 \; 2), (3 \; 4)\rangle.
\]
In both cases $(1 \; 2)(3\; 4)\in G$, which acts freely on the set $\{1, 2, 3, 4\}$. 
Thus the criteria of Theorem~\ref{thm:sparsecubics} is satisfied.
\end{comment}

\medskip

The cases $n=2$ and $n=1$ are similar, but more elementary, and so we omit the details.
We observe however that the case $n=1$ is already treated in Theorem~\ref{thm:original}.
\begin{comment}
We now suppose $n=2$. After possible swapping $\varepsilon$ and $\delta$, we may suppose 
that $\varepsilon$ is a root of one of the polynomials~\eqref{eqn:minpol2}. 
To complete the proof, it suffices to show the criteria of Theorem~\ref{thm:sparsecubics} is 
satisfied, where $F$ is one of 
\begin{align*}
\begin{split}
P(T,X) & = X^3 + T X^2 - (2T + 3) X - 1;\\
Q(T,X) & = X^3 + T X^2 - (2T + 4) X + 1;\\
R(T,X) & = X^3 + T X^2 - (2T + 5) X + 1.
\end{split}
\end{align*}
In the notation of Theorem~\ref{thm:sparsecubics}, we have $\Theta=\{-2,\; -3\}$ by Lemma~\ref{lem:minpol2}. 
There are three possibilities for the discriminant of $F$ with respect to $X$ given by
\begin{align*}
\begin{split}
\Delta_{X}(P) & = (4T^2 + 24 T + 9)(T+3)^2;\\
\Delta_{X}(Q) & = 4 T^4 + 44 T^3 + 172 T^2 + 312 T + 229;\\
\Delta_{X}(R) & = 4 T^4 + 48 T^3 + 229 T^2 + 510 T + 473.
\end{split}
\end{align*}
Thus in the notation of Theorem~\ref{thm:sparsecubics}, 
either $g = 4T^2 + 24 T + 9$ or $g = \Delta_{X}(Q)$ or $g = \Delta_{X}(R)$. 
In all three cases, it is straightforward to check that $g$ is irreducible. 
Again, we conclude that $\Gal(g)$ contains an element 
acting freely on the roots of $g$ by~\cite{Jordan}. 

The $n=1$ case is similar and, moreover, is given by Theorem~\ref{thm:original}.
\end{comment}
This completes the proof of Theorem~\ref{thm:main}. We refer the reader to 
\[
	\text{\url{https://github.com/MaleehaKhawaja/Unit}} 
\]
for the supporting Magma code.
\bibliographystyle{abbrv}
\bibliography{Unit}
\end{document}